\documentclass[11pt]{amsart}

\usepackage{hyperref}
\usepackage{amssymb,amsmath}
\usepackage[alphabetic,nobysame,abbrev]{amsrefs}
\usepackage{enumerate}
\usepackage{color}
\usepackage{pdfsync}
\RequirePackage{fix-cm}
\usepackage{bm}
\usepackage{mdframed}
\usepackage{cancel}
\usepackage{makecell}
\usepackage{pdfpages}
\usepackage{pgf,tikz}
\usepackage{mathrsfs}
\usetikzlibrary{arrows}
\usetikzlibrary{arrows,decorations.markings}
\usetikzlibrary{matrix}
\usetikzlibrary{backgrounds}

\definecolor{qqqqff}{rgb}{0.,0.,1.}
\definecolor{xdxdff}{rgb}{0.49019607843137253,0.49019607843137253,1.}

\definecolor{qqqqff}{rgb}{0.,0.,1.}

\usepackage{tikz}
\usepackage{graphicx}
\usepackage[cmtip,all]{xy}

\usepackage[draft]{todonotes}
\usepackage[cmtip,all]{xy}

\usepackage{hyperref}
\usepackage[utf8]{inputenc}
\usepackage{amsmath,amsfonts,amssymb,euscript,graphicx,color,tikz}

\theoremstyle{plain}
\newtheorem{theorem}{Theorem}[section]
\newtheorem{thm}[theorem]{Theorem}
\newtheorem{lem}[theorem]{Lemma}

\newtheorem{pro}[theorem]{Proposition}

\theoremstyle{definition}
\newtheorem{DEF}[theorem]{Definition}

\newtheorem{exa}[theorem]{Example}

\newtheorem{con}[theorem]{Convention}
\newtheorem{rem}[theorem]{Remark}

\newtheorem{parag}[theorem]{{}}

\numberwithin{equation}{section}

\newcommand{\sub}{\subseteq}

\newcommand{\fm}{(\cdot,\cdot)}
\newcommand{\fh}{\mathfrak{h}}

\newcommand{\K}{\mathbb{K} }

\newcommand{\Z}{\mathbb{Z} }

\newcommand{\ffg}{\mathfrak{g}}

\newcommand{\ep}{\hfill$\Box$}

\def\ad{\hbox{ad}}
\def\Ad{\hbox{Ad}}

\def\sg{\sigma}
\def\a{\alpha}
\def\b{\beta}
\def\lam{\lambda}
\def\Lam{\Lambda}
\def\ep{\epsilon}

\def\supp{\hbox{supp}}
\def\id{\hbox{id}}
\def\Aut{\hbox{Aut}}

\def\fm{(\cdot,\cdot)}
\def\a{\alpha}

\def\w{{\mathcal W}}

\def\sub{\subseteq}

\def\lam{\lambda}
\def\Lam{\Lambda}

\def\1k{\frac{1}{k}}

\def\d{\delta}

\def\b{\beta}

\def\sg{\sigma}

\def\sg{\sigma}

\usepackage{verbatim} 

\def\ad{\hbox{ad}}

\def\bq{{\bf q}}
\def\bbbc{{\mathbb C}}
\def\bbbz{{\mathbb Z}}
\def\Z{{\mathbb Z}}

\def\bbbr{{\mathbb R}}
\def\bbbf{{\mathbb F}}

\def\bbbk{{\mathbb K}}

\def\aa{\mathcal A}

\def\cc{{\mathcal C}}

\def\ep{\epsilon}

\def\ll{{\mathcal G }}

\def\ll{\mathcal L}

\def\supp{\hbox{supp}}

\def\bb{{\mathcal B}}

\def\DynkinNodeSize{1.5mm}
\def\DynkinArrowLength{2mm}
\tikzset{
	dnode/.style={
		circle,
		inner sep=0pt,
		minimum size=\DynkinNodeSize,
		fill=white,
		draw},
	middlearrow/.style={
		decoration={markings,
			mark=at position 0.8 with
			{\draw (0:0mm) -- +(+140:\DynkinArrowLength); \draw (0:0mm) -- +(-140:\DynkinArrowLength);},
		},
		postaction={decorate}
	},
	leftrightarrow/.style={
		decoration={markings,
			mark=at position 0.999 with
			{
				\draw (0:0mm) -- +(+135:\DynkinArrowLength); \draw (0:0mm) -- +(-135:\DynkinArrowLength);
			},
			mark=at position 0.001 with
			{
				\draw (0:0mm) -- +(+45:\DynkinArrowLength); \draw (0:0mm) -- +(-45:\DynkinArrowLength);
			},
		},
		postaction={decorate}
	},
	sedge/.style={
	},
	dedge/.style={
		middlearrow,
		double distance=0.6mm,
	},
	tedge/.style={
		middlearrow,
		double distance=1.0mm+\pgflinewidth,
		postaction={draw}, 
	},
	infedge/.style={
		leftrightarrow,
		double distance=0.5mm,
	},
}

\begin{document}

%
%
\title{Involutive root-graded Lie algebras and Lie tori of type $A$}

\author{Saeid Azam, Mehdi Izadi Farhadi}
\address
{Department of Pure Mathematics\\Faculty of Mathematics and Statistics\\
	University of Isfahan\\ P.O.Box: 81746-73441\\ Isfahan, Iran, and\\
	School of Mathematics, Institute for
	Research in Fundamental Sciences (IPM), P.O. Box: 19395-5746.} \email{azam@ipm.ir, azam@sci.ui.ac.ir}
\address
{Department of Pure Mathematics\\Faculty of Mathematics and Statistics\\
	University of Isfahan\\ P.O.Box: 81746-73441\\ Isfahan, Iran, and\\
	School of Mathematics, Institute for
	Research in Fundamental Sciences (IPM), P.O. Box: 19395-5746.}  \email{m.izadi@ipm.ir}
\thanks{This work is based upon research funded by Iran National Science Foundation (INSF) under project No. 4001480.}
\thanks{This research was in part carried out in the IPM-Isfahan Branch.}


\begin{abstract}
    We investigate the concept of a ``Chevalley involution'' within the framework of root-graded Lie algebras with compatible grading. We provide a characterization of all centerless Lie tori of type $A_\ell(\ell\geq2)$ admitting a Chevalley involution. This work extends and completes the earlier results 
    regarding the existence of such involutions for Lie tori of reduced types.
\end{abstract}
\maketitle

\section{Introduction}\label{intro}

The concept of a Lie algebra graded by a finite root system or a root-graded Lie algebra over a field of characteristic zero was introduced by Berman and Moody \cite{BM92}. Let $R$ be an irreducible reduced finite root system. A Lie algebra graded by $R$ or an $R$-graded Lie algebra is a triple $(\ll,\ffg,\fh)$ consisting of a Lie algebra $\ll$, a finite-dimensional simple Lie subalgebra $\ffg$ of $\ll$ and a Cartan subalgebra $\fh$ of $\ffg$, so that $\ffg$ has the root system $R$ with respect to $\fh$, satisfying certain natural conditions (see Definition \ref{root-graded}).
Berman and Moody showed that if $R$ is simply-laced of rank $\geq2$, then there is an algebra $\aa$ attached to $\ll$, called the coordinate algebra of $\ll$, which depending on the type and rank of $R$, is either commutative and associative, only associative, or alternative (see Proposition \ref{thm0}). Their motivation was to study and understand the intersection matrix algebras (IM algebras) arising from multiply-affinized Cartan matrices of simply-laced types. Toward this end, they classified $R$-graded Lie algebras, up to central extensions, when $R$ is simply-laced of rank $\geq2$. Then, they applied this classification to obtain a realization for these IM algebras.

Benkart and Zelmanov continued the study of root-graded Lie algebras in \cite{BZ96}. They completed the classification of these Lie algebras for the remaining reduced types. Using a Jordan theoretical approach, Neher described $R$-graded Lie algebras in a unified way, when $R$ is not of type $E_8,F_4$ or $G_2$ \cite{Neh96}. His results remain valid for root systems of infinite rank and Lie algebras over rings.

It is well known that Lie tori are root-graded Lie algebras. The concept of a Lie torus arises in the study of extended affine Lie algebras (or EALAs in short) \cite{Yos06}, \cite{Neh04-1}, \cite{Neh04}. EALAs are natural higher dimensional generalizations of finite-dimensional simple Lie algebras and affine Kac-Moody Lie algebras \cite{AABGP97}, \cite{Neh11}. This class of Lie algebras has been under intensive investigation for the past few decades. The core of an extended affine Lie algebra is its subalgebra generated by non-isotropic root spaces. It is well understood that an EALA's structure can be comprehended through analyzing and coordinatizing its core (or core modulo the center, called the centerless core). This is the point at which the concept of a Lie torus emerges; the centerless core of any EALA is a centerless Lie torus, and conversely, any centerless Lie torus is the centerless core of an EALA \cite{Yos06}. Using structure theory and the well-known recognition theorems of root-graded Lie algebras, the coordinatization theorems for Lie tori have been proved and the structure of EALAs have been determined in \cite{BGK96},\cite{BGKN95}, \cite{Yos00} and \cite{AG01}. According to the type and rank, the coordinate algebras that occur are certain unital associative, Jordan, or alternative algebras, referred to as tori (see Definition \ref{def7}).

Now, let us clarify the origin of our motivation for this work. A Lie algebra over a field of positive characteristic is called a modular Lie algebra. In the classical Lie theory, the work of Chevalley on integral forms ($\Z$-forms) establishes a framework for the transition from non-modular to modular theory, which led to the construction of Chevalley groups. In this context, Chevalley involutions play an important role. 

Originally, extended affine Lie algebras were defined over the field of complex numbers. Later, a definition of an EALA over a field of characteristic zero was proposed by a slight modification of the defining axioms of a complex EALA \cite{Neh04}. {Additionally, in \cite{MY06} a more general form of an EALA was introduced over a field of characteristic zero, known as a locally extended affine Lie algebra.} In \cite{AFI22}, the authors started studying extended affine Lie algebras from the modular point of view. In particular, {they} proved the existence of certain integral forms for the cores of reduced extended affine Lie algebras of rank at least $2$ that are equipped with a Chevalley involution; an involution that acts as minus identity on the Cartan subalgebra. By analogy with finite and affine cases, this allows to define a counterpart for these involutive extended affine Lie algebras in the modular setting. Here, arises a question; “If any EALA admits a Chevalley involution?”; and if the answer is no; “Which EALAs are involutive?”

We investigated the existence problem of Chevalley involutions for EALAs in \cite{AI23}. To do this, we first introduced the notion of a Chevalley involution for a Lie torus and studied the conditions under which a Chevalley involution on the centerless core of an EALA, characterized to be a centerless Lie torus, can be extended initially to the core and subsequently to the entire EALA naturally. With case-by-case considerations and employing the coordinatization theorems of Lie tori, we then examined the existence of Chevalley involutions for centerless Lie tori of reduced types. It turns out that the Chevalley involutions of Lie tori are tied with the Chevalley involutions of finite-dimensional simple Lie algebras and certain (anti)-involutions of their coordinate algebras (various tori), called pre-Chevalley (anti)-involutions; see Definition \ref{chev-inv}(i). Through this, we proved that almost all centerless Lie tori of reduced types admit a Chevalley involution. The exception is for type $A_\ell (\ell\geq2)$; we showed that a centerless Lie torus of type $A_\ell (\ell\geq2)$ admits a Chevalley involution if its coordinate algebra is equipped with a pre-Chevalley (anti)-involution. So, the class of all such involutive Lie tori for this type remained unidentified. 

In this direction, this study proceeds in two steps. First, we introduce the notion of a Chevalley involution for $R$-graded Lie algebras $\ll=(\ll,\ffg,\fh)$ that admit a compatible grading over a free abelian group $\Lam$. In fact, an $R$-graded Lie algebra $\ll=(\ll,\ffg,\fh)$ is said to be $(R,\Lam)$-graded if $\ll=\bigoplus_{\lam\in\Lam}\ll^\lam$ is $\Lam$-graded, and $\ll^0$ contains the finite-dimensional simple Lie algebra $\ffg$ with Cartan subalgebra $\fh$. This class of Lie algebras was introduced by Yoshii as a generalization of root-graded Lie algebras. He studied predivision $(R,\Lam)$-graded Lie algebras where $\Lam$ is an abelian group, and classified them up to central extensions, when $R$ is simply-laced of rank $\geq3$ and $\Lam=\bbbz^n$ \cite{Yos01}. The cores of EALAs serve as examples of these $(R,\Lam)$-graded Lie algebras. 

We continue to explore how Chevalley involutions relate to the concept of an involution compatible with the root-grading (Definition \ref{cam}), as introduced by Gao for root-graded Lie algebras \cite{Gao96}. In his paper, Gao studied involutive root-graded Lie algebras to identify compact forms of IM algebras. Some of his work could be viewed as a unitary version of \cite{BM92}.

We prove that any Chevalley involution of an $(R,\Lam)$-graded Lie algebra $\ll$ is compatible with the $R$-grading, see Proposition \ref{pro3}. {By applying} the results of \cite[$\S2$]{Gao96} to our {context}, we obtain some primary results on Chevalley involutions of $(R,\Lam)$-graded Lie algebras. In particular, we show that any Chevalley involution of an $(R,\Lam)$-graded Lie algebra induces a pre-Chevalley (anti)-involution on its coordinate algebra when $R$ is simply-laced of rank $\geq2$, see Proposition \ref{pro1}.  

In the second step of this work, as a special subclass of root-graded Lie algebras with compatible grading, we examine centerless Lie tori of type $A_\ell$, for $\ell\geq2$. It turns out that there exist certain Lie tori of this type which do not admit any Chevalley involution. We characterize all Lie tori of type $A_\ell (\ell\geq2)$ admitting a Chevalley involution, see Theorem \ref{thm5}. The coordinatization theorem for Lie tori of this type is essential in our characterization; see Theorem \ref{thm4} and Theorem \ref{thm3}. In fact, considering Proposition \ref{pro1}, we identify the coordinate algebras of centerless Lie tori of type $A_\ell (\ell\geq2)$ admitting a pre-Chevalley (anti)-involution, which are the octonion torus and elementary quantum tori (for definition of these tori see \ref{ex1} and \ref{def5}, respectively). This finalizes the investigation of Chevalley involutions for Lie tori of reduced types, as initiated in \cite{AI23}.

The organization of the paper is as follows. In Section \ref{sec2} basic definitions and notation from the theory of root-graded Lie algebras are collected. In addition, we record some necessary background and results from \cite{BM92} needed to describe the coordinate algebras of these Lie algebras. Then we recall the definition of a root-graded Lie algebra with compatible grading; see Definition \ref{root-graded-com}. Here, we introduce the notion of a Chevalley involution for this class of Lie algebras; see Definition \ref{chev-inv}. At the end of the section, Lie tori are defined as a subclass of root-graded Lie algebras with compatible grading, see Definition \ref{Lie-torus}.

Section \ref{sec3} discusses the involutions of an $R$-graded Lie algebra that are compatible with the $R$-grading, see Definition \ref{cam}. Here, some required results from \cite[$\S2$]{Gao96} regarding these involutions are restated. 
We show that Chevalley involutions for an $(R,\Lam)$-graded Lie algebra are compatible with the $R$-grading, see Proposition \ref{pro3}. This enables us to deduce certain immediate consequences concerning Chevalley involutions of $(R,\Lam)$-graded Lie algebras, which will be essential for the proof of our main result in the next section.

Finally, in Section \ref{sec4}  by revisiting the coordinatization theorems related to Lie tori of type $A_\ell$, $\ell\geq 2$, we determine all Lie tori of these types that possess a Chevalley involution, as presented in Theorem \ref{thm5}. This completes the characterization of Lie tori of reduced types admitting a Chevalley involution, paving the way for further study of the modular theory of extended affine Lie algebras.

\section{Preliminaries}\label{sec2}

{Throughout this work $\bbbk$ is a field of characteristic $0$. All vector spaces and algebras are assumed to be over $\bbbk$ and all algebras (except Lie algebras) are assumed to be unital.} If $X$ is a subset of a group, we denote by $\langle X\rangle$ the subgroup generated by $X$. For a positive integer $n$, we denote by $\bbbk[x_1^{\pm1},\ldots,x_n^{\pm1}]$ the associative algebra of Laurent polynomials over $\K$ in $n$ commuting variables $x_1,\ldots,x_n$.
For a Lie algebra $\ll$, by $Z(\ll)$ we mean the center of $\ll$. Also, by a centerless Lie algebra, we mean a Lie algebra with a trivial center.

In this work, we always assign the symbol $R$ to denote an irreducible reduced finite root system. By $R^\times$, we mean $R\setminus\{0\}$.

\subsection{Root-graded Lie algebras}
Here we consider Lie algebras graded by the root lattice of a finite root system, the so-called root-graded Lie algebras introduced by S. Berman and R. V. Moody \cite{BM92}.
We begin by recalling the definition of a root-graded Lie algebra.

\begin{DEF}\label{root-graded}
    Let $\ffg$ be a finite-dimensional simple Lie algebra with Cartan subalgebra $\fh$ and $R$ be the root system of $\ffg$ with respect to $\fh$. A Lie algebra $\ll=(\ll,\ffg,\fh)$ is called {\it graded by} $R$ or {\it $R$-graded} with {\it grading pair} $(\ffg,\fh)$ if
	
    (RG1) $\ll$ contains $\ffg$ as a subalgebra,
	
    (RG2) $\ll=\bigoplus_{\a\in R}\ll_\a$, where $\ll_\a=\{x\in\ll\mid[h,x]=\a(h)x,\text{for all}\;h\in\fh\}$,
	
    (RG3) $\ll_0=\sum_{\a\in R^\times}[\ll_\a,\ll_{-\a}]$.\\
    We will use the abbreviation $\ll$ for $(\ll,\ffg,\fh)$ when there is no confusion.
\end{DEF}

Let $\ll$ be an $R$-graded Lie algebra with grading pair $(\ffg,\fh)$. We have the usual root space decomposition $\ffg=\bigoplus_{\a\in R}\ffg_\a$  with $\ffg_\a\sub\ll_\a$ for $\a\in R$. We fix a Chevalley basis $\{e_\a,h_i\mid e_\a\in\ffg_\a,\a\in R^\times,h_i\in\fh,1\leq i\leq\ell\}$ of $\ffg$, where $\ell=\dim\fh$. Let $G=\langle\exp te_\a\mid t\in\K,\a\in R^\times\rangle$ be the corresponding universal Chevalley group and set
$$\eta_\a(t):=(\exp te_\a)(\exp(-t^{-1}e_{-\a}))(\exp te_\a),$$
$$\mathcal{N}:=\langle \eta_\a(t)\mid t\in\K^\times,\a\in R^\times\rangle,$$
$$\mathcal{T}:=\langle \eta_\a(t)\eta_\a(1)^{-1}\mid t\in\K^\times,\a\in R^\times\rangle.$$
It is well-known that $\mathcal{T}$ is a normal subgroup of $\mathcal{N}$ and $\mathcal{N}/\mathcal{T}\cong\w$, where $\w$ is the Weyl group of $ R$. We consider the group homomorphism 
$$\begin{array}{c}
\Ad:G\rightarrow\Aut\ll\vspace{2mm}\\
\exp te_\a\mapsto\exp\ad te_\a.
\end{array}$$

We now assume that $ R$ is of type $A_\ell(\ell\geq2),D_\ell(\ell\geq4),E_{6,7,8}$. Here, we briefly review some results from \cite[\S 1]{BM92} that will be needed later. Let $\a,\b\in R^\times$. One can choose $w\in\w$ so that $w(\a)=\b$. Let $\eta\in\mathcal{N}$ with $\eta\mathcal{T}\leftrightarrow w$ under the isomorphism $\mathcal{N}/\mathcal{T}\cong\w$. Considering $\ll$ as an integrable $\ffg$-module via adjoint representation, then $\Ad(\eta)(\ll_\a)=\ll_\b$. Let $\tilde\theta_{\b,\a}:=\Ad(\eta)|_{\ll_\a}$. We note that $\tilde\theta_{\b,\a}$ is determined by $\a,\b$ up to a non-zero scalar multiple and does not depend on the choice of $w$ or $\eta$. Since $\eta(\ffg_\a)=\ffg_\b$, then $\eta(e_\a)=\ep e_\b$, for some $\ep\in\K^\times$. Set \begin{equation}\label{lam}
\theta_{\b,\a}:=\ep^{-1}\tilde\theta_{\b,\a}:\ll_\a\rightarrow\ll_\b.
\end{equation}
So $\theta_{\b,\a}$ depends only on $\a,\b$ and the choice of the Chevalley basis. {We have}
\begin{eqnarray*}	
    &&\theta_{\a,\b}=\theta^{-1}_{\b,\a},\\	&&\theta_{\a,\b}\theta_{\b,\gamma}=\theta_{\a,\gamma},\\
    &&\theta_{\a,\a}=\id.
\end{eqnarray*}

For a fixed $\a\in R^\times$, let $\aa=\ll_\a$ as a $\K$-vector space. We proceed to put an algebra structure on $\aa$, not depending on the particular root $\a$. Thus, for any $a \in \aa$, we denote it by $e_\a(a)$ when considering it as an element of $\ll_\a$. Now, let $\b\in R^\times$ and define $e_\b(a)=\theta_{\b,\a}{(e_\a(a))\in\ll_\b}$.

One can define a multiplication on $\aa$ via the notion of an $A_2$-pair. Recall that a pair of linearly independent roots $\a,\b\in R^\times$ is called an {\it $A_2$-pair} if $(\bbbz\a\oplus\bbbz\b)\cap R$ forms an $A_2$-finite root system in $\bbbr\a\oplus\bbbr\b$. Now fix an $A_2$-pair {$(\b,\gamma)$} and define a multiplication $m_{(\b,\gamma)}:\aa\times\aa\rightarrow\aa$ by
\begin{equation}\label{mul}
[e_\b(a),e_\gamma(b)]:=[e_\b,e_\gamma](m_{(\b,\gamma)}(a,b)),
\end{equation}
for $a,b\in\aa$. {Note that the $A_2$-pair $(\gamma,\b)$ defines the opposite multiplications to $m_{(\b,\gamma)}$, that is, $m_{(\gamma,\b)}(a,b)=m_{(\b,\gamma)}(b,a)$ for all $a,b\in\aa$ (see \cite[1.18]{BM92}).} Multiplication (\ref{mul}) makes $\aa$ into a $\K$-algebra. It is called the {\it coordinate algebra (ring of coordinates)} of $\ll$. Moreover:

\begin{thm}[{\cite[1.19, 3.17]{BM92}}]\label{thm0}
    The $\K$-algebra $\aa$ is
    \begin{enumerate}[\upshape (i)]
        \item associative if $R$ is of rank at least $3$;
        
        \item commutative if $R$ is of type $D$ or $E$;

        \item alternative if $R$ is of type $A_2$.
    \end{enumerate}
    In addition, when $ R$ is of type $D$ or $E$, then the multiplication is independent of the choice of $A_2$-pair. If $ R$ is of type $A$, then the two equivalence classes of $A_2$-pairs define opposite multiplications on $\aa$.
\end{thm}

We recall that two perfect Lie algebras are said to be centrally isogeneous if their universal covering algebras are isomorphic.

\begin{pro}[{\cite[1.29]{BM92}}]
    Let $\ll_1$ and $\ll_2$ be centrally isogeneous Lie algebras. Assume that $\ll_1$ is graded by $R$. Then $\ll_2$ is also $ R$-graded and in such a way that the associative coordinate algebras are isomorphic.
\end{pro}

\subsection{Chevalley involution for $(R,\Lam)$-graded Lie algebras}
Let $\Lam$ be a free abelian group and as before, let $R$ be an irreducible reduced finite root system. First, we recall from \cite{Yos01} the notion of an $(R,\Lam)$-graded Lie algebra.

\begin{DEF}\label{root-graded-com}
    We say an $R$-graded Lie algebra $\ll$ with grading pair $(\ffg,\fh)$ admits a compatible $\Lam$-grading or $\ll$ is {\it $(R,\Lam)$-graded} if $\ll=\bigoplus_{\lam\in\Lam}\ll^\lam$ is a $\Lam$-graded Lie algebra so that $\ffg\subseteq\ll^0$ and $\langle\supp(\ll)\rangle=\Lam$, where $\supp(\ll)=\{\lam\in\Lam\mid\ll^\lam\neq\{0\}\}$.
\end{DEF}

Note that since $\ffg\subseteq\ll^0$, each $\ll^\lam$ is an $\fh$-module and then we get $\ll=\bigoplus_{\a\in R,\lam\in\Lam}\ll^\lam_\a$, where $\ll^\lam_\a:=\ll^\lam\cap\ll_\a$ for $\lam\in\Lam$ and $\a\in R$.

\begin{DEF}\label{def1}
    An $(R,\Lam)$-graded Lie algebra $\ll$ is called

    \begin{enumerate}[\upshape (i)]
        \item {\it predivision} if for each $\a\in R^\times$ {with} $\ll_\a^\lam\neq(0)$, there exist $e\in\ll_\a^\lam$ and $f\in\ll_{-\a}^{-\lam}$ such that
        $$[[e,f],x_\b]=\langle\b,\a^\vee\rangle x_\b,$$
        for $\b\in R,x_\b\in\ll_\b$. Here $\a^\vee$ is the coroot of $\a$ in the sense of \cite{Bou68} and $\langle\b,\a^\vee\rangle$ is the corresponding Cartan integer; 
        
        \item {\it division} if for each $\a\in R^\times,\lam\in\Lam$ and $0\neq e\in\ll_\a^\lam$ there exists $f\in\ll_{-\a}^{-\lam}$ such that
        $$[[e,f],x_\b]=\langle\b,\a^\vee\rangle x_\b,$$
        for $\b\in R,x_\b\in\ll_\b$.
    \end{enumerate}
\end{DEF}

\begin{rem}\label{rem1}
    (i) We note that in Definition \ref{def1}, (ii) implies (i), namely any division $(R,\Lam)$-graded Lie algebra is predivision. Also, when $\dim_\K(\ll_\a^\lam)\leq1$ for all $\a\in R^\times$, these two concepts coincide.

    (ii) In \cite[Proposition 2.13]{Yos01} for predivision $(R,\Lam)$-graded Lie algebras of simply-laced type with $\text{rank} R\geq3$, and in \cite[Proposition 6.3]{Yos02} for division $(A_2,\Lam)$-garded Lie algebras, Yoshii showed that their coordinate algebras are naturally $\Lam$-graded. Moreover, he classified these Lie algebras up to central extensions when $\Lam=\Z^n$ (see \cite[\S 4, Corollary]{Yos01} and \cite[Theorem 6.4]{Yos02}).
\end{rem}

\begin{exa}\label{ex2}
    (i) Any $R$-graded Lie algebra $\ll$ is a predivision $(R,\Lam)$-graded Lie algebra with $\Lam=\{0\}$.

    (ii) The (centerless) core of an extended affine Lie algebra $E$ of reduced type $R$ with nullity $n$ is a division $(R,\Z^n)$-graded Lie algebra; for details see \cite[Example 2.8 (b)]{Yos01}.

    (iii) Assume that $\aa=\bigoplus_{\lam\in\Lam}\aa^\lam$ is a $\Lam$-graded commutative associative algebra with $\langle\supp(\aa)\rangle=\Lam$, and $\ffg$ is a finite-dimensional simple Lie algebra with Cartan subalgebra $\fh$ and root system $R$. Let $\ll:=\ffg\otimes_\K\aa$. Then $\ll=\bigoplus_{\a\in R}(\ffg_\a\otimes_\K\aa)$ is an $R$-graded Lie algebra with grading pair $(\ffg\otimes1,\fh\otimes1)$. Next, let $\ll^\lam:=\ffg\otimes_\K\aa^\lam$, for all $\lam\in\Lam$.
    Then $\ll=\bigoplus_{\lam\in\Lam}\ll^\lam$ is a $\Lam$-graded Lie algebra such that $\ffg\otimes1\sub\ll^0$, and $\supp(\ll)=\supp(\aa)$. Thus, the $\Lam$-grading is compatible and makes $\ll$ an $(R,\Lam)$-graded Lie algebra.
\end{exa}

{Let $\aa$ be an algebra over $\bbbk$. In the following, by an {\it (anti)-involution} $\bar{\:\:}$ of $\aa$ we mean a $\bbbk$-linear map $\bar{\:\:}:\aa\rightarrow\aa$ satisfying: 
\begin{eqnarray*}	
    &&\overline{ab}=\overline{b}\overline{a},\\	
    &&\overline{\overline{a}}=a,
\end{eqnarray*}
for all $a,b\in\aa$. In addition, an automorphism $\tau$ of a Lie algebra $\ll$ over $\bbbk$ is called an {\it involution} if $\tau^2=\id$.} In \cite{AI23}, we introduced the notion of a Chevalley involution for a Lie torus to study Chevalley involutions of extended affine Lie algebras. Here, we generalize this notion for an $(R,\Lam)$-graded Lie algebra; see Remark \ref{rem2}.

\begin{DEF}\label{chev-inv}
    Let $\aa=\bigoplus_{\lam\in\Lam}\aa^\lam$ be a $\Lam$-graded algebra and $\ll=(\ll,\ffg,\fh)$ be an $(R,\Lam)$-graded Lie algebra.
    
    \begin{enumerate}[\upshape (i)]
        \item We call an (anti)-involution $\bar{\:\:}$ of $\aa$, a {\it pre-Chevalley (anti)-involution}, if $\overline{\aa^\lam}=\aa^{-\lam}$ for all $\lam\in\Lam$;
        
        \item An involution $\tau$ of $\ll$ is called a {\it Chevalley involution}, if $\tau$ preserves $\ffg$, $\tau(\ll^\lam)=\ll^{-\lam}$ for all $\lam\in\Lam$, and $\tau(x)=-x$ for all $x\in\ll^0_0$.
    \end{enumerate}
\end{DEF}

\begin{exa}\label{ex3}
    (i) Let $\aa$, $\ffg$ and $\ll=\ffg\otimes_\K\aa$ be as in Example \ref{ex2}(iii). We further suppose that $\aa$ is equipped with a pre-Chevalley (anti)-involution $\bar{\:\:}$ and $\sg$ is a Chevalley involution of $\ffg$. We define $\tau:\ll\rightarrow\ll$ by 
    $$\tau(x\otimes a)=\sg(x)\otimes\bar{a},$$
    for all $x\in\ffg$ and $a\in\aa$. Clearly $\tau$ is an involution of $\ll$ with $\tau(\ll^\lam)=\ll^{-\lam}$, for all $\lam\in\Lam$. Moreover, $\tau(\ffg\otimes1)=\ffg\otimes1$ and $\tau_{|_{\ll^0_0}}=-\id$. So $\tau$ is a Chevalley involution of $\ll=\ffg\otimes_\K\aa$.

    (ii) Assume that $\bbbf$ is a field extension of $\K$ and let $\aa={\bbbf}[x_1^{\pm1},\ldots,x_n^{\pm1}]$ be the algebra of Laurent polynomials over $\bbbf$ in $n$ variables. We note that $\aa=\bigoplus_{\lam\in\Z^n}\bbbf x^\lam$ is a $\Z^n$-graded algebra, where $x^\lam=x_1^{\lam_1}\cdots x_ n^{\lam_ n}$ for $\lam=(\lam_1,\ldots,\lam_n)\in\Z^n$. Then $\ll=\ffg\otimes_\K\aa$ is a division $(R,\Z^n)$-graded Lie algebra; see \cite[\S 4, Corollary(ii)]{Yos01}. One easily sees that the assignment $x_j\mapsto x_j^{-1},1\leq j\leq n$, induces a pre-Chevalley (anti)-involution $\bar{\:\:}$ on $\aa$. Thus, as defined in part (i), $\tau$ is a Chevalley involution of $\ll=\ffg\otimes_\K \bbbf[x_1^{\pm1},\ldots,x_n^{\pm1}]$.
\end{exa}

Assume that $\ll=(\ll,\ffg,\fh)$ is an $(R,\Lam)$-graded Lie algebra. Let $\tau$ be a Chevalley involution of $\ll$. Since $\fh\subseteq\ll_0^0$, then $\tau_{|_\fh}=-\id$. So $\tau_{|_{\ffg}}$ is a Chevalley involution of the finite-dimensional simple Lie algebra $\ffg$. Next, one can choose a Chevalley basis $\{e_\a,h_i\mid e_\a\in\ffg_\a,\a\in R^\times,h_i\in\fh,1\leq i\leq\ell\}$ of $\ffg$ so that $\tau(e_\a)=-e_{-\a}$.

\begin{con}
    In the remaining, whenever $\ll$ is equipped with a Chevalley involution $\tau$, we fix the Chevalley basis of $\ffg$ corresponding to $\tau_{|_{\ffg}}$, chosen as above.
\end{con}

\subsection{Lie tori}
Yoshii introduced the notion of a Lie torus in \cite{Yos06} to characterize extended affine Lie algebras. In fact, the centerless core of any EALA is a centerless Lie torus, and
conversely, any centerless Lie torus is the centerless core of an EALA. Neher further studied Lie tori in \cite{Neh04-1}.

\begin{DEF}\label{Lie-torus}
    A division $(R,\Z^n)$-graded Lie algebra $\ll$ is called a {\it Lie $(R,\Z^n)$-torus} or a {\it Lie $n$-torus} if $\dim_\K(\ll_\a^\lam)\leq1$ for all $\a\in R^\times$ and $\lam\in\Z^n$. When $R$ and $n$ are fixed, we simply say that $\ll$ is a Lie torus, and call $n$ the \textit{nullity} of $\ll$.
\end{DEF}

The Lie torus $\ll$ is called {\it invariant}, if $\ll$ has an invariant non-degenerate symmetric bilinear form $\fm$ which is graded, meaning that
$$ (\ll^\lam_\a,\ll^\mu_\b)=0\quad\hbox{unless}\quad(\a,\lam)=(-\b,-\mu).$$

\begin{exa}
    In Example \ref{ex3}, let $\aa=\K[x_1^{\pm1},\ldots,x_n^{\pm1}]$. Then $\ll=\ffg\otimes_\K\aa$ is a Lie $(R,\Z^n)$-torus with Chevalley involution $\tau$.
\end{exa}

Based on the structure theory of root-graded Lie algebras, the coordinatization theorems for centerless Lie tori of reduced types have been established in \cite{BGK96},\cite{BGKN95}, \cite{Yos00} and \cite{AG01}. In essence, a centerless Lie $n$-torus can be formed as a matrix Lie algebra coordinatized by an $n$-torus. Depending on the type, it can be a unital associative, Jordan, or alternative algebra. We will now review the definition of an $n$-torus.

\begin{DEF}\label{def7}
    Let $\aa=\bigoplus_{\lam\in\Z^n}\aa^\lam$ be a $\Z^n$-graded algebra over $\bbbk$. $\aa$ is called an {\it $n$-torus} if

    \begin{enumerate}[\upshape (i)]
        \item $\dim_\bbbk \aa^\lam\leq1$, for all $\lam\in\Z^n$, and each non-zero $x\in \aa^\lam$ is invertible,
        
        \item $\supp_\Lam(\aa)$ generates $\Lam$, where $\supp_\Lam(\aa)=\{\lam\in\Z^n\mid \aa^\lam\neq\{0\}\}$.
    \end{enumerate}
    If an $n$-torus $\aa$ is a Jordan algebra, alternative algebra, or associative algebra then $\aa$ is called a {\it Jordan, alternative, or associative $n$-torus}, respectively.
\end{DEF}

\begin{rem}
    In \cite{AI23}, we studied the Chevalley involutions for Lie tori of reduced types. For more examples of Chevalley involutions of Lie tori, we refer the reader to this article. There, we observed that the Chevalley involutions of Lie tori are tied with the Chevalley involutions of finite-dimensional simple Lie algebras and the pre-Chevalley (anti)-involution of their coordinate algebras which are $n$-tori. Through this, we showed the existence of Chevalley involutions for all centerless Lie tori of reduced types except for type $A_\ell$, $\ell\geq 2$. The main goal of this work is to investigate the case $A_\ell$, $\ell\geq 2$, which will be addressed in the final section.
\end{rem}

\section{Compatible involutions}\label{sec3}

Assume as before that $R$ is an irreducible reduced finite root system and $\Lam$ is a free abelian group. First, we recall from \cite{Gao96} the notion of an involution of an $R$-graded Lie algebra $\ll$ compatible with the $R$-grading. Then we show that any Chevalley involution of an $(R,\Lam)$-graded Lie algebra $\ll$ is compatible with its $R$-grading. Consequently, when $R$ is simply-laced of rank $\geq2$, applying the results of \cite[$\S2$]{Gao96} to our setting we get that any Chevalley involution of an $(R,\Lam)$-graded Lie algebra induces a pre-Chevalley (anti)-involution on its coordinate algebra. 
 
\begin{DEF}\label{cam}
Let $\ll$ be an $R$-graded Lie algebra. An involution $\sg$ of $\ll$ is called {\it compatible} with the $ R$-grading if
	
    \begin{enumerate}[\upshape (i)]
        \item $\sg(\ll_\a)\subseteq\ll_{-\a}$, and
        
        \item $\sg\Ad\eta_\a(1)=\Ad\eta_\a(1)\sg$,
    \end{enumerate}
for all $\a\in R^\times$.
\end{DEF}

\begin{lem}\label{151}
    Assume that $\ll=(\ll,\ffg,\fh)$ is an $(R,\Lam)$-graded Lie algebra with an involution $\tau$ such that $\tau(\ll^\lam)=\ll^{-\lam}$ for all $\lam\in\Lam$, and $\tau_{|_{\ll_0^0}}=-\id$. Then $\tau(\ll_\a)=\ll_{-\a}$, for all $\a\in R$.
\end{lem}
\begin{proof}
    Note that $\fh\subseteq\ll_0^0$ and since $\tau_{|_{\ll_0^0}}=-\id$ we have
    $$[h,\tau(x)]=\tau([-h,x])=-{\a}(h)\tau(x),$$
    for $\a\in {R},x\in\ll_\a$ and $h\in\fh$. So $\tau(\ll_\a)=\ll_{-\a}$, for $\a\in R$.
\end{proof}

\begin{pro}\label{pro6}
    Let $\ll=(\ll,\ffg,\fh)$ be an $(R,\Z^n)$-graded Lie torus with an involution $\tau$ such that $\tau(\ll^\lam)=\ll^{-\lam}$ for all $\lam\in\Lam$, and $\tau_{|_{\ll_0^0}}=-\id$. Then $\tau(\ffg)\subseteq\ffg$. In particular, $\tau$ is a Chevalley involution of $\ll$.
\end{pro}
\begin{proof}
    Since $\ffg_\a\subseteq\ll_\a^0$ and $\dim_\K(\ll_\a^0)\leq1$, for all $\a\in R^\times$, it follows from Lemma \ref{151} that $\tau(\ffg)\subseteq\ffg$ and so $\tau$ is a Chevalley involution of $\ll$.
\end{proof}

\begin{rem}\label{rem2}
    We note that by Proposition \ref{pro6}, Definition \ref{chev-inv} for a Lie torus $\ll$ coincides with \cite[Definition 3.1.1]{AI23}.
\end{rem}

\begin{pro}\label{pro3}
    Let $\ll=(\ll,\ffg,\fh)$ be an $(R,\Lam)$-graded Lie algebra with a Chevalley involution $\tau$. Then $\tau$ is campatible with the $ R$-grading.
\end{pro}
\begin{proof}
By Lemma \ref{151}, we have $\tau(\ll_\a)=\ll_{-\a}$, for all $\a\in R$, so it remains to show that condition (ii) of Definition \ref{cam} holds for $\tau$. To complete the proof, we now mimic the argument given in \cite[Lemma 2.3]{Gao96}. First, we note that if $\phi$ is an automorphism of $\ll$, then $\phi\exp\ad x\phi^{-1}=\exp\ad(\phi(x))$
for all ad-nilpotent $x\in\ll$. Then we have
\begin{eqnarray*}
    \tau\Ad\eta_\a(1)\tau^{-1}&=&\tau(\exp\ad e_\a)\tau^{-1}\tau(\exp\ad(-e_{-\a}))\tau^{-1}\tau(\exp\ad e_\a)\tau^{-1}\\
    &=&(\exp\ad(\tau(e_\a)))(\exp\ad(\tau(-e_{-\a})))(\exp\ad(\tau(e_\a)))\\
    &=&(\exp\ad(-e_{-\a}))(\exp\ad e_\a)(\exp\ad(-e_{-\a}))\\
    &=&(\exp\ad(-e_{-\a}))\Ad\eta_\a(1)(\exp\ad(-e_{\a}))\\
    &=&(\exp\ad(-e_{-\a}))(\exp\ad(\Ad\eta_\a(1)(-e_{\a})))\Ad\eta_\a(1).
\end{eqnarray*}
On the other hand
\begin{eqnarray*}
    \Ad\eta_\a(1)(-e_{\a})&=&(\exp\ad e_\a)(\exp\ad(-e_{-\a}))(\exp\ad e_\a)(-e_{\a})\\
    &=&(\exp\ad e_\a)(\exp\ad(-e_{-\a}))(-e_{\a})\\
    &=&\exp\ad e_\a(-e_\a-h_\a+e_{-\a})\\
    &=&-e_\a+(-h_\a+2e_\a)+(e_{-\a}+h_\a-e_\a)\\
    &=&e_{-\a}.
\end{eqnarray*}
Thus, we get $\tau\Ad\eta_\a(1)\tau^{-1}=\Ad\eta_\a(1)$, as required.
\end{proof}

The following example shows that the converse of the previous result is not generally valid; specifically,  any involution of $\ll$ that is compatible with the $R$-grading is not necessarily a Chevalley involution.

\begin{exa}
    In Example \ref{ex2}(iii), let $\aa=\K[x_1^{\pm1},\ldots,x_n^{\pm1}]$ and consider $e_\a\in\ffg_\a$, $\a\in R^\times$, such that $\{e_\a\mid\a\in R^\times\}$ extends to a Chevalley basis for $\ffg$. Next, let $\bar{\:\:}:\aa\rightarrow\aa$ be the (anti)-involution induced by $x_j^{\pm 1}\mapsto -x^{\pm 1}_j,1\leq j\leq n$. Also, let $\sg$ be the involution on $\ffg$ induced by $\sg(e_\a)=-e_{-\a}$, $\a\in R^\times$. Then the assignment 
    $$\tau:e_\a\otimes a\mapsto\sg(e_{\a})\otimes\bar{a},\qquad (\a\in R^\times,a\in\aa)$$
    induces an involution of $\ll$ with $\tau(\ll_\a)=\ll_{-\a}$. Also, by the argument given in the proof of Proposition \ref{pro3}, $\tau$ commutes with all $\Ad\eta_\a(1)$, for $\a\in R^\times$. So $\tau$ is compatible with the $R$-grading. But $\tau$ is not a Chevalley involution for $\ll$, since $\tau(\ll^\lam)=\ll^\lam$, for all $\lam\in\Lam$.
\end{exa}

We proceed with recalling what we will need from \cite[$\S2$]{Gao96}. Assume that $ R$ is a simply-laced irreducible finite root system of rank $\geq2$. Let $\ll$ be an $ R$-graded Lie algebra with a compatible involution $\sg$ and $\aa$ be the coordinate algebra of $\ll$. As we have already seen, $\aa$ can be identified with $\ll_\a$ for any $\a\in R^\times$. We fix $\a\in R^\times$ and define a $\K$-linear map $\bar{\:\:}:\aa\rightarrow\aa$ by
\begin{equation}\label{anti}
	e_{-\a}(\bar a)=-\sg(e_\a(a)),
\end{equation}
for all $a\in\aa$. In fact, $\bar{\;}$ is independent of the choice of $\a$.
Since $\sg$ is an involution, $\bar{\bar a}=a$ for all $a\in\aa$. Moreover:

\begin{pro}[{\cite[Proposition 2.27]{Gao96}}]\label{pro5}
    The $\K$-linear map $\bar{\;}$ is an (anti)-involu\-tion of the $\K$-algebra $\aa$.
\end{pro}

\begin{rem}
    We note that the aforementioned result is actually founded on Berman and Moody's work, which includes type $A_2$. But in his article, Gao excluded this type. However, Proposition \ref{pro5} applies to type $A_2$.
\end{rem}

\begin{lem}\label{lem1}
    Let $\ll=(\ll,\ffg,\fh)$ be an $(R,\Lam)$-graded Lie algebra, where $ R$ is a simply-laced irreducible finite root system of rank $\geq2$. Then, for $\a,\b\in R^\times$ the map $\theta_{\b,\a}:\ll_\a\rightarrow\ll_\b$ defined by (\ref{lam}) preserves the $\Lam$-grading, i.e. $\theta_{\b,\a}(\ll_\a^\lam)=\ll_\b^\lam$, for all $\lam\in\Lam$.
\end{lem}
\begin{proof}
   An argument similar to \cite[Proposition 1.27]{AABGP97} or \cite[Theorem 5.1]{Yos04} shows that $$\Ad\eta_\delta(1)(\ll_{\gamma}^\lam)=\ll^{\lam}_{w_{\delta}(\gamma)},\qquad(\d,\gamma\in R^\times,\lam\in\Lam),$$
    where $w_{\delta}$ is the reflection based on $\delta$, as an element of the Weyl group of $ R$. Now, one clearly observes from the definition of $\theta_{\b,\a}$ that $\theta_{\b,\a}$ maps $\ll_{\a}^\lam$ onto $\ll_{\b}^\lam$.
\end{proof}

In the following, we obtain a necessary condition for a simply-laced $(R,\Lam)$-graded Lie algebra $\ll$ to admit a Chevalley involution. In the next section, this is essential for the proof of Theorem \ref{thm5} concerning involutive Lie tori of type $A$. Here, we need the coordinate algebra of $\ll$ to be a $\Lam$-graded algebra. So in this result, we consider division $(R,\Lam)$-graded Lie algebras; see Remark \ref{rem1}(ii).

\begin{pro}\label{pro1}
    Let $\ll=(\ll,\ffg,\fh)$ be a {division} $(R,\Lam)$-graded Lie algebra with a Chevalley involution $\tau$, where $ R$ is a simply-laced irreducible finite root system of rank $\geq2$. Then the coordinate algebra $\aa$ of $\ll$ admits a pre-Chevalley (anti)-involution.
\end{pro}
\begin{proof}
   By Proposition \ref{pro3}, $\tau$ is compatible with the $R$-grading. Then Proposition \ref{pro5} implies that the coordinate algebra $\aa$ has an (anti)-involution $\bar{\;}$. We fix $\a\in R^\times$ and identify $\aa$ with $\ll_\a$. Recall from (\ref{anti}) that $\bar{\;}:\aa\rightarrow\aa$ is given by $e_{-\a}(\bar a)=-\tau(e_\a(a))$, for all $\a\in\aa$. Equivalently $\bar a=-\theta_{\a,-\a}(\tau(e_\a(a)))$.
    
    We now show that $\bar{\;}:\aa\rightarrow\aa$ maps $\aa^\lam$ onto $\aa^{-\lam}$, for all $\lam\in\Lam$, i.e., $\bar{\;}$ is a pre-Chevalley (anti)-involution of $\aa$. Note that the involution $\tau$ maps $\aa^\lam=\ll_\a^\lam$ onto $\ll_{-\a}^{-\lam}$, for all $\lam\in\Lam$. On the other hand, we know from Lemma \ref{lem1} that $\theta_{\a,-\a}:\ll_{-\a}\rightarrow\ll_\a$ maps $\ll_{-\a}^{-\lam}$ onto $\aa^{-\lam}=\ll_\a^{-\lam}$. Therefore, $\overline{\aa^\lam}=\aa^{-\lam}$, as desired.
\end{proof}

\section{Involutive Lie tori of type $A_\ell (\ell\geq2)$}\label{sec4}

In this section, we examine Lie tori of type $A_\ell (\ell\geq2)$. In \cite[Theorem 5.0.1]{AI23}, we showed that a centerless Lie torus of type $A_\ell (\ell\geq2)$ admits a Chevalley involution if its coordinate algebra is equipped with a pre-Chevalley (anti)-involution.  Here, our aim is to identify all centerless Lie tori of type $A_\ell (\ell\geq2)$ admitting a Chevalley involution. The coordinatization theorem for Lie tori of this type is essential in this identification. We begin by revisiting the coordinatization theorems for Lie tori of type $A_\ell (\ell\geq2)$.

\begin{DEF}
    Assume that $\aa$ and $\bb$ are $\Lam$-graded algebras. We say that $\aa$ and $\bb$ are graded-isomorphic, if there is an algebra isomorphism $\varphi:\aa\rightarrow\bb$ that preserves the $\Lam$-grading.
\end{DEF}

{Berman, Gao, and Krylyuk over the field of complex numbers $\bbbc$ and Yoshii in general proved the coordinatization theorem for Lie tori of type $A_\ell,\ell\geq3$. Moreover, Berman, Gao, Krylyuk, and Neher over $\bbbc$ and Yoshii in general obtained a similar result for Lie tori of type $A_2$.}

\begin{thm}\label{thm4}
    Assume that $\ll$ is a centerless Lie $n$-torus of type $A_\ell (\ell\geq2)$.

    \begin{enumerate}[\upshape (i)]
        \item \upshape{(\cite[Theorem 2.65]{BGK96},\cite[\S4]{Yos01})} If $\ell\geq3$, then $\ll$ is graded-isomorphic to $\mathfrak{sl}_{\ell+1}(\aa)$ where $\aa$ is an associative $n$-torus.
        
        \item \upshape{(\cite[Lemma 3.25]{BGKN95},\cite[Proposition 6.3]{Yos02})} If $\ell=2$, then $\ll$ is graded-isomorphic to $\mathfrak{psl}_3(\aa)$ where $\aa$ is an alternative $n$-torus.
    \end{enumerate}
\end{thm}

The coordinate algebras of Lie tori of type $A_\ell (\ell\geq3)$ which are associative tori by Theorem \ref{thm4}, are characterized as ``quantum tori''.

\parag
{(Quantum torus).}\label{def5}
    Let $\bq=(q_{ij})\in M_ n(\bbbk)$ be a {{\it quantum matrix}, namely} an $( n\times n)$-matrix such that 
    \begin{equation}\label{eq1}
        q_{ii}=1,1\leq i\leq n, \quad\text{and}\quad q_{ij}q_{ji}=1,1\leq i\neq j\leq n.
    \end{equation}
   The {\it quantum torus} $\bbbk_\bq=\bbbk_\bq[x_1^{\pm1},\ldots,x_n^{\pm1}]$ based on a quantum matrix $\bq$ is defined as the unital associative algebra generated by $2 n$ elements $x_1^{\pm1},\ldots,x_ n^{\pm1}$ with the defining relations
    \begin{equation}\label{eq2}
        x_i x_i^{-1}=1=x_i^{-1}x_i,1\leq i\leq n, \quad\text{and}\quad x_i x_j=q_{ij}x_j x_i,1\leq i\neq j\leq n.
    \end{equation}
    Notice that if $q_{ij}=1$ for all $i,j$, then $\bbbk_\bq=\bbbk[x_1^{\pm1},\ldots,x_n^{\pm1}]$ becomes the algebra of Laurent polynomials in $n$ variables over $\K$. Let $\{\varepsilon_1,\ldots,\varepsilon_n\}$ be a $\bbbz$-basis of $\Z^n$. Then $\bbbk_\bq$ is $\Z^n$-graded with $\bbbk_\bq=\sum_{\lam\in\Lam}\bbbk x^\lam$, where for $\lam=\lam_1\varepsilon_1+\cdots+\lam_n\varepsilon_n\in\Z^n$, $x^\lam := x_1^{\lam_1}\cdots x_ n^{\lam_ n}$. This $\Z^n$-grading turns $\K_\bq$ into an associative $n$-torus.

\begin{DEF}
    A quantum matrix $\bq=(q_{ij})$ is called {\it elementary} when $q_{ij}=\pm1$ for all $i,j$. The quantum torus $\bbbk_\bq$ based on an elementary quantum matrix $\bq$ is called an {\it elementary quantum torus}.
\end{DEF}

We note that any associative $n$-torus is also an alternative $n$-torus. However, not all alternative tori are associative. The octonion torus, initially discovered in \cite{BGKN95}, serves as an example of a nonassociative alternative torus. This torus is constructed as an octonion algebra using the \textit{Cayley-Dickson} process over an algebra of Laurent polynomials. Y. Yoshii provided a straightforward description of the octonion torus through a presentation in \cite{Yos08}, as detailed below.

\parag
{(The octonion torus).}\label{ex1}
     Consider the alternative algebra $\cc$ over $\K$ with generators $x_i^{\pm1}$ for $1\leq i\leq3$, {and defining relations} $x_ix_i^{-1}=1=x_i^{-1}x_i$ for all $i$, $x_ix_j=-x_jx_i$ for $i\neq j$, and $(x_1x_2)x_3=-x_1(x_2x_3)$. For {$n\geq 3$}, define 
     $$\mathbb{O} := \cc \otimes \K[x_4^{\pm1},\ldots,x_n^{\pm1}],$$
     where $\K[x_4^{\pm1}, \ldots, x_n^{\pm1}]$ is the algebra of Laurent polynomials in $n-3$ variables over $\K$. Let $\{\epsilon_1,\ldots,\epsilon_n\}$ be the standard basis of $\Z^n$. Then $\mathbb{O}$ is an alternative $n$-torus with the $\Z^n$-grading given by $\deg(x_i) = \epsilon_i$. The alternative $n$-torus $\mathbb{O}$ is called the {\it octonion $n$-torus}.

The following description of alternative tori was provided in \cite[Theorem 1.25]{BGKN95} {and \cite[corollary 5.13]{Yos02}.}

\begin{thm}\label{thm3}
    Let $\aa$ be an alternative $n$-torus over $\K$.

    \begin{enumerate}[\upshape (i)]
        \item If $\aa$ is associative, then $\aa$ is graded-isomorphic to the quantum torus $\K_\bq$ based on some quantum matrix $\bq=(q_{ij})\in M_ n(\bbbk)$.
        
        \item If $\aa$ is not associative, then $\aa$ is graded-isomorphic to the octonion $n$-torus $\mathbb{O}$.
    \end{enumerate}
\end{thm}

We conclude this section with the following theorem, which characterizes Lie tori that admit a Chevalley involution.
The description of coordinate algebras for Lie tori of type $A_\ell (\ell\geq2)$ given in Theorem \ref{thm3} is essential in our characterization.

\begin{thm}\label{thm5}
    Assume that $\ll$ is a centerless Lie $n$-torus of type $A_\ell, l\geq2$ and $\aa=\bigoplus_{\lam\in\Z^n}\aa^\lam$ is the coordinate algebra of $\ll$. Then $\ll$ admits a Chevalley involution if and only if $\aa$ admits a pre-Chevalley (anti)-involution, if and only if the coordinate algebra $\aa$ {is graded-isomorphic to either the octonion $n$-torus or an elementary quantum torus.}
\end{thm}
\begin{proof}
Assume that $\aa$ is equipped with a pre-Chevalley (anti)-involution. Then it follows from the proofs of \cite[Propositions 5.1.3 and 5.2.3]{AI23} that $\ll$ admits a Chevalley involution. Conversly, if $\ll$  admits a Chevalley involution then by Proposition \ref{pro1}, $\aa$ is equipped with a pre-Chevalley (anti)-involution. We now consider the second ``if and only if" in the statement.

We know from Theorems \ref{thm4} and \ref{thm3} that the coordinate algebra $\aa$ is graded-isomorphic to either the octonion $n$-torus $\mathbb{O}$ ($\ell=2$), or the quantum torus $\K_\bq$ based on some quantum matrix $\bq=(q_{ij})\in M_ n(\bbbk)$. So we may assume $\aa=\mathbb{O}$ or $\aa=\K_\bq$. By the first ``if and only if", we need to prove that $\aa$ is equipped with a pre-Chevalley (anti)-involution if and only if $\aa\neq\K_\bq$ where $\bq$ is non-elementary.

If $\aa=\mathbb{O}$, then there exists an (anti)-involution $\sg$ on $\aa$ determined by $x_i\mapsto-x_i$ for $1\leq i\leq3$, and $x_j\mapsto x_j$ for $4\leq j\leq n$, see \cite[Lemma 1.20]{BGKN95}. Also, $\aa$ is equipped with an automorphism $\tau$ given by $\tau(x_i)=x^{-1}_i$ for $1\leq i\leq n$, with $\tau^2=\id$, and $\tau(\aa^\lam)=\aa^{-\lam}$ for all $\lam\in\Lam$, see \cite[Proposition 4.2.4]{AI23}. Let $\bar{\;}=\sg\tau$. Since $\sg\tau=\tau\sg$, it follows that $\bar{\;}$ is an (anti)-involution of $\aa=\mathbb{O}$ with $\overline{\aa^\lam}=\aa^{-\lam},\lam\in\Lam$, that is, $\bar{\;}$ is a pre-Chevalley (anti)-involution of $\aa$.

Next, let $\aa=\bbbk_{\bq}$. Assume that $\bar{\;}:\aa\rightarrow\aa$ is a pre-Chevalley (anti)-involution. Since $\overline{\aa^\lam}=\aa^{-\lam}$ for $\lam\in\Lam$, we have $\overline{x_i}=k_ix^{-1}_i$ for some $k_i\in\K^\times$, $1\leq i\leq n$. By the definition of $\K_\bq$, we have $x_ix_j=q_{ij}x_jx_i$ for $1\leq i\neq j\leq n$. Now applying $\bar{\;}$ to this relation, we get $k_ik_jx^{-1}_jx^{-1}_i=q_{ij}k_ik_jx^{-1}_ix^{-1}_j$, or equivalently $x_jx_i=q_{ij}x_ix_j$. Thus, $x_jx_i=q_{ij}x_ix_j=q^2_{ij}x_jx_i$.
This gives $q_{ij}=q_{ji}=\pm1,1\leq i\neq j\leq n$, that is, $\bq$ is elementary. Conversely, assume that $\bq$ is elementary. Since the elements $x^{\mp1}_i,1\leq i\leq n$, of the opposite algebra $\K_\bq^{\text{op}}$ satisfy the defining relations (\ref{eq2}), we get an (anti)-homomorphism $\bar{\;}:\K_\bq\rightarrow\K_\bq$ satisfying $\overline{x_i}=x^{-1}_i,1\leq i\leq n$. The map $\bar{\;}$ is clearly a bijection due to the fact that $\dim_\bbbk \aa^\lam=1$ for all $\lam\in\Lam$. Thus, $\bar{\;}$ is an (anti)-involution on $\aa=\K_\bq$ with $\overline{\aa^\lam}=\aa^{-\lam}$, for all $\lam\in\Lam$, that is, $\bar{\;}$ is a pre-Chevalley (anti)-involution of $\aa$. This concludes the proof.
\end{proof}

\begin{rem}
    {Yoshii showed that a centerless Lie torus of type $A_1$ can be constructed using the Tits-Kantor-Koecher (TKK) construction from a Jordan torus \cite[Theorem 1]{Yos00}. Furthermore, he described five families of Jordan $n$-tori and then proved that every Jordan $n$-torus is graded-isomorphic to a torus belonging to one of these families \cite[Theorem 2]{Yos00}. The simplest is the family consisting of the Jordan tori $\K_\bq^+$, which $\K_\bq^+$ denotes the $\Z^n$-graded algebra $\K_\bq$ with multiplication $x\cdot y =1/2(xy+yx)$. In \cite[Proposition 5.3.4]{AI23}, we showed the existence of Chevalley involutions for all Lie tori of type $A_1$. In particular, the Lie torus $\text{TKK}(\K_\bq^+)$ admits a Chevalley involution for any quantum torus $\K_\bq$. Thus, in the sense of Theorem \ref{thm5}, type $A_1$ is considered special.}
\end{rem}

\begin{rem}
    {We note that Theorem \ref{thm5} completes the investigation of the existence of Chevalley involutions for Lie tori of reduced types, as initiated in \cite{AI23}. It seems that studying the existence problem of Chevalley involutions for the non-reduced case, that is, type $BC$ requires engaging with various structural theories and results presented in \cite{AY03}, \cite{AFY08}, \cite{Fau08} and \cite{AB10}. We plan to address this problem in future work.}
\end{rem}



\begin{bibdiv}
	\begin{biblist}
		
		\bib{AABGP97}{article}{
		      label={AABGP97},
			author={{Allison}, Bruce},
			author={{Azam}, Saeid},
			author={{Berman}, Stephen},
			author={{Gao}, Yun},
			author={{Pianzola}, Arturo},
			title={{Extended affine Lie algebras and their root systems}},
			date={1997},
			ISSN={0065-9266; 1947-6221/e},
			journal={{Mem. Am. Math. Soc.}},
			volume={603},
			pages={122},
		}

            \bib{AB10}{article}{
                author={Allison, Bruce},
                author={Benkart, Georgia},
                isbn={978-0-8218-4507-3},
                book={
                title={Quantum affine algebras, extended affine Lie algebras, and their applications. Proceedings of the workshop, March 2--7, 2008, Banff, Canada},
                publisher={Providence, RI: American Mathematical Society (AMS)},
                },
                title={Unitary Lie algebras and Lie tori of type {{\(\text{BC}_r\)}}, {{\(r\geq 3\)}}},
                pages={1--47},
                date={2010},
            }   

		\bib{AG01}{article}{
			author={Allison, Bruce},
			author={Gao, Yun},
			title={The root system and the core of an extended affine {Lie} algebra.},
			date={2001},
			ISSN={1022-1824},
			journal={Sel. Math., New Ser.},
			volume={7},
			number={2},
			pages={149\ndash 212},
		}

            \bib{AFY08}{article}{
                author={Allison, Bruce},
                author={Faulkner, John},
                author={Yoshii, Yoji},
                issn={0092-7872},
                issn={1532-4125},
                title={Structurable tori},
                journal={Commun. Algebra},
                volume={36},
                number={6},
                pages={2265--2332},
                date={2008},
                publisher={Taylor \& Francis, Philadelphia, PA},
            }

            \bib{AY03}{article}{
                author={Allison, Bruce},
                author={Yoshii, Yoji},
                issn={0022-4049},
                issn={1873-1376},
                title={Structurable tori and extended affine Lie algebras of type BC{{\(_{1}\)}}},
                journal={J. Pure Appl. Algebra},
                volume={184},
                number={2-3},
                pages={105--138},
                date={2003},
                publisher={Elsevier (North-Holland), Amsterdam},
            }
		
		\bib{AFI22}{article}{
                label={AFI22},
			author={Azam, Saeid},
			author={Farahmand~Parsa, Amir},
			author={Izadi~Farhadi, Mehdi},
			title={Integral structures in extended affine {Lie} algebras},
			date={2022},
			ISSN={0021-8693},
			journal={J. Algebra},
			volume={597},
			pages={116\ndash 161},
		}

            \bib{AI23}{article}{
                label={AI23},
			author={Azam, Saeid},
			author={Izadi~Farhadi, Mehdi},
			title={Chevalley involutions for {Lie} tori and extended affine {Lie}
				algebras},
			date={2023},
			ISSN={0021-8693},
			journal={J. Algebra},
			volume={634},
			pages={1\ndash 43},
		}

		\bib{BZ96}{article}{
			author={Benkart, Georgia},
			author={Zelmanov, Efim},
			title={Lie algebras graded by finite root systems and intersection matrix algebras},
			date={1996},
			ISSN={0020-9910},
			journal={Invent. Math.},
			volume={126},
			number={1},
			pages={1\ndash 45},
		}
        
		\bib{BGK96}{article}{
			author={Berman, Stephen},
			author={Gao, Yun},
			author={Krylyuk, Yaroslav~S.},
			title={Quantum tori and the structure of elliptic quasi-simple {Lie}
				algebras},
			date={1996},
			ISSN={0022-1236},
			journal={J. Funct. Anal.},
			volume={135},
			number={2},
			pages={339\ndash 389},
			url={semanticscholar.org/paper/ec0c3961f4e077104c5d099700fb975dd7aa01e4},
		}
		
		\bib{BGKN95}{article}{
			author={Berman, Stephen},
			author={Gao, Yun},
			author={Krylyuk, Yaroslav},
			author={Neher, Erhard},
			title={The alternative torus and the structure of elliptic quasi-simple
				{Lie} algebras of type {{\(A_ 2\)}}},
			date={1995},
			ISSN={0002-9947},
			journal={Trans. Am. Math. Soc.},
			volume={347},
			number={11},
			pages={4315\ndash 4363},
		}
		
		\bib{BM92}{article}{
			author={Berman, S.},
			author={Moody, R.~V.},
			title={Lie algebras graded by finite root systems and the intersection
				matrix algebras of {Slodowy}},
			date={1992},
			ISSN={0020-9910},
			journal={Invent. Math.},
			volume={108},
			number={2},
			pages={323\ndash 347},
		}
		
		\bib{Bou68}{misc}{
			author={Bourbaki, N.},
			title={{Groupes} et alg{\`e}bres de {Lie}. Ch. {VI}},
			publisher={Hermann, Paris},
			date={1968},
		}

            \bib{Fau08}{article}{
                author={Faulkner, John},
                issn={0092-7872},
                issn={1532-4125},
                title={Lie tori of type {{\(BC_{2}\)}} and structurable quasitori},
                journal={Commun. Algebra},
                volume={36},
                number={7},
                pages={2593--2618},
                date={2008},
                publisher={Taylor \& Francis, Philadelphia, PA},
            }
        
		\bib{Gao96}{article}{
			author={Gao, Yun},
			title={Involutive {Lie} algebras graded by finite root systems and
				compact forms of {IM} algebras},
			date={1996},
			ISSN={0025-5874},
			journal={Math. Z.},
			volume={223},
			number={4},
			pages={651\ndash 672},
			url={https://eudml.org/doc/174944},
		}

            \bib{MY06}{article}{
                 author={Morita, Jun},
                 author={Yoshii, Yoji},
                 issn={0021-8693},
                 issn={1090-266X},
                 title={Locally extended affine Lie algebras},
                 journal={J. Algebra},
                 volume={301},
                 pages={59--81},
                 date={2006},
                }

		\bib{Neh04}{article}{
			author={{Neher}, Erhard},
			title={{Extended affine Lie algebras}},
			date={2004},
			ISSN={0706-1994},
			journal={{C. R. Math. Acad. Sci., Soc. R. Can.}},
			volume={26},
			number={3},
			pages={90\ndash 96},
		}
		
		\bib{Neh04-1}{article}{
			author={{Neher}, Erhard},
			title={{Lie tori}},
			date={2004},
			ISSN={0706-1994},
			journal={{C. R. Math. Acad. Sci., Soc. R. Can.}},
			volume={26},
			number={3},
			pages={84\ndash 89},
		}
		
		\bib{Neh11}{incollection}{
			author={{Neher}, Erhard},
			title={{Extended affine Lie algebras and other generalizations of affine
					Lie algebras -- a survey}},
			date={2011},
			booktitle={{Developments and trends in infinite-dimensional Lie theory}},
			publisher={Basel: Birkh\"auser},
			pages={53\ndash 126},
		}
		
		\bib{Neh96}{article}{
			author={Neher, Erhard},
			title={Lie algebras graded by 3-graded root systems and {Jordan} pairs
				covered by grids},
			date={1996},
			ISSN={0002-9327},
			journal={Am. J. Math.},
			volume={118},
			number={2},
			pages={439\ndash 491},
			url={muse.jhu.edu/journals/american_journal_of_mathematics/toc/ajm118.2.html},
		}
		
		\bib{Yos00}{article}{
			author={Yoshii, Yoji},
			title={Coordinate algebras of extended affine {Lie} algebras of type
				{{\(A_1\)}}},
			date={2000},
			ISSN={0021-8693},
			journal={J. Algebra},
			volume={234},
			number={1},
			pages={128\ndash 168},
		}
		
		\bib{Yos01}{article}{
			author={Yoshii, Yoji},
			title={Root-graded {Lie} algebras with compatible grading},
			date={2001},
			ISSN={0092-7872},
			journal={Commun. Algebra},
			volume={29},
			number={8},
			pages={3365\ndash 3391},
		}
		
		\bib{Yos02}{article}{
			author={Yoshii, Yoji},
			title={Classification of division {{\({\mathbb Z}^n\)}}-graded
				alternative algebras},
			date={2002},
			ISSN={0021-8693},
			journal={J. Algebra},
			volume={256},
			number={1},
			pages={28\ndash 50},
		}
		
		\bib{Yos04}{article}{
			author={Yoshii, Yoji},
			title={Root systems extended by an {Abelian} group and their {Lie}
				algebras},
			date={2004},
			ISSN={0949-5932},
			journal={J. Lie Theory},
			volume={14},
			number={2},
			pages={371\ndash 394},
			url={https://eudml.org/doc/125971},
		}
		
		\bib{Yos06}{article}{
			author={{Yoshii}, Yoji},
			title={{Lie tori -- a simple characterization of extended affine Lie
					algebras}},
			date={2006},
			ISSN={0034-5318; 1663-4926/e},
			journal={{Publ. Res. Inst. Math. Sci.}},
			volume={42},
			number={3},
			pages={739\ndash 762},
		}
		
		\bib{Yos08}{article}{
			author={Yoshii, Y.},
			title={Cayley polynomials},
			date={2008},
			ISSN={0373-9252},
			journal={Algebra Logika},
			volume={47},
			number={1},
			pages={54\ndash 70},
		}
		
	\end{biblist}
\end{bibdiv}
\end{document}